\documentclass[preprint,3p,times]{elsarticle}
\usepackage{amsmath}
\usepackage{mathrsfs}
\usepackage{tikz}
\usepackage{hyperref}
\usepackage{listings}
\usepackage{multirow}
\usepackage{amssymb}
\usepackage{amsthm}
\usepackage{epstopdf}
\usepackage{algorithm}
\usepackage{algpseudocode}
\usepackage{subfigure}
\usepackage{tabularx,booktabs}
\usepackage{caption}
\usepackage{extarrows}
\usepackage{comment}
\usepackage{xcolor}

\biboptions{sort&compress}
\theoremstyle{plain}
  \newtheorem{theorem}{Theorem}[section]
  
  \newtheorem{proposition}{Proposition}[section]
  \newtheorem{corollary}{Corollary}[section]
\theoremstyle{definition}
  \newtheorem{definition}{Definition}[section]
  
  \newtheorem{example}{Example}[section]
  \newtheorem{remark}{Remark}[section]

\makeatletter
\newif\if@borderstar
\def\bordermatrix{\@ifnextchar*{%
\@borderstartrue\@bordermatrix@i}{\@borderstarfalse\@bordermatrix@i*}%
}
\def\@bordermatrix@i*{\@ifnextchar[{\@bordermatrix@ii}{\@bordermatrix@ii[()]}}
\def\@bordermatrix@ii[#1]#2{%
\begingroup
\m@th\@tempdima8.75\p@\setbox\z@\vbox{%
\def\cr{\crcr\noalign{\kern 2\p@\global\let\cr\endline }}%
\ialign {$##$\hfil\kern 2\p@\kern\@tempdima &\thinspace %
\hfil $##$\hfil &&\quad\hfil $##$\hfil\crcr\omit\strut %
\hfil\crcr\noalign{\kern -\baselineskip}#2\crcr\omit %
\strut\cr}}%
\setbox\tw@\vbox{\unvcopy\z@\global\setbox\@ne\lastbox}%
\setbox\tw@\hbox{\unhbox\@ne\unskip\global\setbox\@ne\lastbox}%
\setbox\tw@\hbox{%
$\kern\wd\@ne\kern -\@tempdima\left\@firstoftwo#1%
\if@borderstar\kern2pt\else\kern -\wd\@ne\fi%
\global\setbox\@ne\vbox{\box\@ne\if@borderstar\else\kern 2\p@\fi}%
\vcenter{\if@borderstar\else\kern -\ht\@ne\fi%
\unvbox\z@\kern-\if@borderstar2\fi\baselineskip}%
\if@borderstar\kern-2\@tempdima\kern2\p@\else\,\fi\right\@secondoftwo#1$%
}\null \;\vbox{\kern\ht\@ne\box\tw@}%
\endgroup
}
\makeatother
\allowdisplaybreaks
\begin{document}
\begin{frontmatter}

\title{Some properties of fuzzy $t$-norm and vague $t$-norm}
\author{Haohao Wang\fnref{label1}}

\author{Bin Yang\fnref{label1}\corref{cor1}}
\ead{binyang0906@nwsuaf.edu.cn}

\author{Wei Li\fnref{label1}}

 \address[label1]{College of Science, Northwest A \& F University, Yangling, Shaanxi 712100, PR China}
\cortext[cor1]{Corresponding author.}
\begin{abstract}
	Rosenfeld defined a fuzzy subgroup of group $G$ as a fuzzy subset of $G$ with two special conditions attached\cite{Rosenfeld1971Fuzzysubgroups}. In this paper, we introduce the fuzzy $t$-norms and vague $t$-norms. The unit interval with a $t$-norm or a $t$-conorm is a special monoid, so we mainly talk about fuzzy subsets of monoids and vague monoids.
Firstly, we  generalize some properties of $t$-norm to fuzzy $t$-subnorm, so that we can analyze and classify the fuzzy $t$-norms. Further, we explore specific research on these properties of the vague $t$-norm. In addition,  the concept of lattice is introduced, and then the present conclusions are extended to bounded lattices.
Finally, we define the concept of fuzzy monoids by aggregate functions, uninorms, nullnorms and draw the relevant results.	
\end{abstract}


\end{frontmatter}

\section{Introduction}
The fuzzification or relaxation of logical connectives can be  effectively applied in solving practical problems such as imprecision, lack of accuracy, or the presence of noise. Rosenfeld firstly defined a fuzzy subgroup of group $G$ which is  a fuzzy subset of $G$ with two special conditions attached\cite{Rosenfeld1971Fuzzysubgroups}. On this basis, many mathematicians have obtained rich results\cite{Somepropertiesoffuzzygroups, Fuzzygroupsandlevelsubgroups, Fuzzyfunctionsandtheirfundamentalproperties}. Considering that many scholars are very interested in fuzzified algebraic structures, including groups, rings, actions, they begin to search for fuzzy cases of various algebraic structures based on Rosenfeld's method\cite{Boixader2018Fuzzyactions,Rosenfeld1971Fuzzysubgroups}.

In Rosenfeld's related work, only a subset of the group $G$ is fuzzy, but the operations on it are still classical, so Mustafa Demirci proposed the idea of  fuzzifying the operations on the group, which uses tools such as fuzzy equality and fuzzy functions\cite{VagueGroups}.

There is evidence that $t$-norm and $t$-conorm as two kinds of fuzzy logic operations play a crucial role in fuzzy sets theory \cite{Zadeh1965Fuzzysets,Zimmermann1991FuzzySet}. To further enrich the properties of the aggregation operators, Yager and Rybalov proposed the concpet of uninorms \cite{Yager1996Uninormaggregation,Fodor1997Structure}. The identity element of uninorm can take any number in the unit interval, not just zero and one in the case of $t$-norms and $t$-conorms. In addition, Calvo et al. introduced the notion of nullnorms in 2001\cite{Calvo2001Thefunctionalequations}. Uninorms and nullnorms are generalizations of $t$-norms and $t$-conorms, and on the other side, there exists close connection between them.

In this paper, we will develop the fuzzification of $t$-norms and $t$-conorms from two perspectives:
\begin{enumerate}[(1)]
	\item Extending properties of $t$-norms to fuzzy $t$-norms, vague $t$-norms and bounded lattices.
	\item  Generalizing the idea of fuzzy $t$-norms to aggregation functions, uninorms and nullnorms.
\end{enumerate}

The first part of this paper studies the fuzzy $t$-norms and vague $t$-norms, then introduces the fuzzification of properties of $t$-norms, such as strict monotonicity, cancellation law, conditional cancellation law, Archimedean and limit property. In the case of $t$-norms, these properties are closely related to each other. After fuzzification, some inter-pushing relationships are still maintained.  The unit interval with $t$-norm and $t$-conorm is a special monoid. However, little is known about the specific situation of fuzzy $t$-norm. We generalize some properties of $t$-norm to it, so that  we can analyze and classify the fuzzy $t$-subnorm. We also explore specific research on these properties on the vague $t$-norm. And then, the concept of lattice is introduced and the content of the previous two subsections  is generalized to the bounded lattice\cite{boundedlattice}.

In the definition of fuzzy groups given by Rosenfeld, the conjunctive operation $\wedge$ is used, which can be replaced by the fuzzy operator such as $t$-norms to get more flexible definitions of fuzzy subset\cite{Vagueandfuzzyt-normsandt-conorms}. Therefore, the last section uses aggregation functions, uninorms and nullnorms to replace the conjunctive operation $\wedge$ as replacements. Thus, we define the concept of fuzzy monoids by aggregate functions, uninorms, nullnorms and obtain the relevant results.

The remainder of this paper is organized as follows. In Section \ref{SEC2}, some basic concepts and properties of  fuzzy set theory and algebraic structure are introduced. In Section \ref{SEC3},  the fuzzification of properties such as strict monotonicity, cancellation law, conditional cancellation law, Archimedean and limit property  are introduced on the fuzzy $t$-norms. In Section \ref{SEC4}, vague $t$-norms and their properties are proposed. Section \ref{SEC5} introduces the concept of lattice, and further generalize the content of previous two sections to bounded lattice. Section \ref{SEC6} defines the concept of fuzzy monoids by aggregate functions, uninorms, nullnorms and obtains relevant results. Finally, we state some conclusions in section \ref{SEC7}.

\section{Preliminaries}\label{SEC2}

This section recapitulates some well-known concepts that shall be used in the sequel.

\begin{definition}\cite{tnorm}\label{d:$t$-norm}
	A $t$-norm is a binary operation $T:[0,1]^2\to [0,1]$, for all $x,y,z\in [0,1]$, which satisfies the following conditions:
	\begin{enumerate}[(T1)]
		\item Commutativity:
		$T(x,y)=T(y,x)$;
		\item Associativity:
		$T(T(x,y),z)=T(x,T(y,z))$;
		\item Monotonicity:
		$T$ is non-decreasing in each argument;
		\item Boundary condition:
		$T(x,1)=x.$
	\end{enumerate}	 	
\end{definition}
If we focus on the algebraic structure, the binary operation $t$-norm with prefix expression $T$  can also be expressed by the infix binary operator $*$, then the four axioms (T1)-(T4) can be followed as:
\begin{align*}
	&	x*y=y*x;\\
	&	x*y\leq x*z, \text{if }  y\leq z;\\
	&	(x*y)*z= x*(y*z);\\
	&	1*x=x.
\end{align*}

\begin{remark}
	The associativity (T2) allows us to extend each $t$-norm $T$ in a unique way to a $n$-ary operation in the usual way by induction, defining for each $n$-tuple $(x_1,x_2,\cdots,x_n) \in [0,1]^n$
	$$T(x_1,x_2,\cdots,x_n)=T(T(x_1,x_2,\cdots,x_{n-1}),x_n).$$
	If, in particular, we have $x_1=x_2=\cdots =x_n=x$, we shall briefly write
	$$x_T^{(n)}=T(x,x,\cdots,x).$$
	Finally we put, by convention, for each $ x \in [0,1]$
	$$x_T^{(0)}=1 \text{ and } x_T^{(1)}=x.$$
\end{remark}

\begin{example}\label{e:$t$-norm}
	The following shows four common $t$-norms:
	\begin{enumerate}[(1)]
		\item
		Minimum: $T_M(x, y) = \textnormal{min}(x, y)$,
		\item
		Product: $T_P(x, y) = xy$,
		\item
		{\L}ukasiewicz $t$-norm: $T_L(x, y) = \textnormal{max}(x+y-1,0)$,
		\item
		Drastic product:\\
		$T_D(x,y) =\left\{
		\begin{aligned}
			&0,&if\ x,y\in [0,1), \\
			&\textnormal{min}(x,y),&otherwise.
		\end{aligned}
		\right.$
	\end{enumerate}
\end{example}

%
%
%

\begin{definition}\cite{tnorm}\label{d:t-conorm}
	A $t$-conorm is a binary operation $S:[0,1]^2\to [0,1]$, for all $x,y,z\in [0,1]$, which satisfies the following conditions:
	\begin{enumerate}[(S1)]
		\item  Commutativity:
		$S(x,y)=S(y,x)$;
		\item Associativity:
		$S(S(x,y),z)=S(x,S(y,z))$;
		\item Monotonicity:
		$S$ is non-decreasing in each argument;
		\item  Boundary condition:
		$S(x,0)=x.$	
	\end{enumerate}	
\end{definition}

\begin{example}\label{e:$t$-norm}
	The following shows four common $t$-conorms:
	\begin{enumerate}[(1)]
		\item
		Maximum: $S_M(x, y) = \textnormal{max}(x, y)$,
		\item
		Probabilistic sum: $S_P(x, y) = x+y-xy$,
		\item
		{\L}ukasiewicz $t$-conorm: $S_L(x, y) = \textnormal{min}(x+y,1)$,
		\item
		Drastic sum:\\
		$S_D(x,y) =\left\{
		\begin{aligned}
			&1&if\ x,y\in (0,1], \\
			&\textnormal{max}(x,y)&otherwise.
		\end{aligned}
		\right.$
	\end{enumerate}
\end{example}

\begin{definition}\cite{tnorm}
	For an arbitrary $t$-norm $T$, we consider the following properties for all $x,y,z\in [0,1]$:
	\begin{enumerate}[(1)]
		\item   $T$ is said to be strictly monotone if
		$$ T(x,y)<T(x,z) \text{  whenever } x>0\text{ and }y<z.$$
		\item   $T$ satisfies the cancellation law if
		$$T(x,y)=T(x,z)\Rightarrow x=0 \text{ or } y=z.$$
		\item   $T$ satisfies the conditional cancellation law if
		$$T(x,y)=T(x,z)>0 \Rightarrow y=z.$$
		\item  $T$ is called Archimedean if
		$$ \forall (x,y)\in(0,1)^2, \text{ there is an } n\in \mathbb{N} \text{ with }\\ x_T^{(n)}<y.$$
		\item  $T$ has the limit property if
		$$\text{for all } x\in(0,1) , \lim\limits_{n \to \infty} x_T^{(n)}=0.$$
	\end{enumerate}
\end{definition}

\begin{example}
	Notice that $T_L$, $T_P$ are Archimedean $t$-norm, and $T_M$ is not.
\end{example}
In what follows, some algebraic concepts are introduced.
\begin{definition}\cite{VagueGroups}
	Let $X$ be a non-empty set, $\circ$ be an operation on $X$.
	\begin{enumerate}[(1)]
		\item  $X$ together with a binary operation $\circ$, denoted by $ (X,\circ)$ is a semigroup if and only if the following associative property is satisfied, i.e.,
		$$a\circ( b\circ c)=(a\circ b)\circ c, \forall a,b,c \in X.$$
		\item
		A semigroup $(X,\circ)$ is a monoid if and only if  there exists an element $e \in X$, called the (two-sided) identity	element of $(X,\circ)$, such that $e\circ a=a$  and $a\circ e=a$ for each $a \in X$.
		\item	A monoid $(X,\circ)$ is a group if and only if for each $a \in X$, there exists an element of $X$, denoted by $a^{-1}$ and called the (two sided) inverse element of a, such that $a^{-1}\circ a=e$ and $a\circ a^{-1} =e.$
		\item A semigroup $(X,\circ)$ is said to be abelian or commutative if and only if the binary
		operation $\circ$ has the following property:
		$$a\circ b=b\circ a ,\forall a,b\in X.$$
	\end{enumerate}
\end{definition}
From an algebraic point of view, a $t$-norm is actually a commutative monoid with identity element $1$.

\begin{definition}\cite{tnorm}
	An aggregation function is a mapping $A:[0,1]^n\rightarrow [0,1]$, $n$ is a positive integer greater than $1$, such that,
	\begin{enumerate}[(1)]
		\item $A(x_1,\cdots,x_n) \leq A(y_1,\cdots,y_n),\text{ whenever } x_i\leq y_i,\text{ for  all } i\in {1,\cdots,n}.$
		\item $A(0,\cdots,0) = 0 \text{ and } A(1,\cdots, 1) = 1.$
	\end{enumerate}	
\end{definition}
\begin{definition}\label{def3}\cite{tnorm}
	A uninorm is a binary function $U : [0, 1]^2 \rightarrow [0, 1]$, for all $x, y, z \in [0, 1]$, which satisfies the following conditions :
	\begin{enumerate}[(U1)]
		\item  Commutativity:
		$U(x, y) = U(y, x)$;
		\item Associativity:
		$U(U(x, y), z) = U (x,U(y, z))$;
		\item Monotonicity:
		$U$ is non-decreasing in each place;
		\item  Identity element:
		$U(x, e) = x, e\in [0,1]$.	
	\end{enumerate}	
\end{definition}

\begin{remark}
	In particular, a uninorm is a $t$-norm when $e=1$ and a $t$-conorm when $e=0$. For any $e\in (0,1)$, uninorm is equivalent to a $t$-norm in $[0,e]^2$, a $t$-conorm in $[e,1]^2$, and its values in $A(e) = [0,e)\times (e,1] \cup (e,1] \times [0,e)$ have the following form:
	$$\textnormal{min}(x,y)\leqslant U(x,y)\leqslant \textnormal{max}(x,y).$$
\end{remark}

\begin{remark}
	A uninorm $U$ is called conjunctive if $U(1, 0)=0$ and a uninorm $U$ is called disjunctive if $U(1,0)=1$. A conjunctive (resp. disjunctive) uninorm $U$ is said to be locally internal on the boundary if it
	satisfies $U (1, x) \in \{1, x\}$ (resp. $U(0, x) \in \{0, x\}$) for all $x \in [0,1]$.
\end{remark}
\begin{example}
	The following are several common unnorms:
	\begin{enumerate}[(1)]
		\item
		$\mathscr{U}_{min}$, $\mathscr{U}_{max}$: those given by minimum and maximum in $A(e)=\{(x,y):\text{min}(x,y)<e<\text{max}(x,y)\}$, respectively.
		\item
		$\mathscr{U}_{ide}$: $U(x, x)=x$ for all $x \in [0, 1]$.
		\item
		$\mathscr{U}_{rep}$: those that have an additive generator.
		\item
		$\mathscr{U}_{cos}$: Continuous in the open square $(0,1)^2$.
	\end{enumerate}
\end{example}

\begin{theorem}\cite{Fodor1997Structure}
	Let $U:[0, 1]^2\rightarrow [0, 1]$ be a uninorm with neutral element $e \in (0,1)$. Then the sections $x \hookrightarrow U(x, 1)$ and $x \hookrightarrow U(x, 0)$ are continuous in each point except perhaps for $e$ if and only if $U$ is given by one of the following formulas.
	\begin{enumerate}[(1)]
		\item
		If $U(0, 1)= 0$, then
		\begin{align*}
			\small
			U(x,y)=\left\{
			\begin{aligned}
				&eT(\frac{x}{e}, \frac{y}{e})&if(x,y)\in [0,e]^2;\\
				&e+(1-e)S(\frac{x-e}{1-e}, \frac{y-e}{1-e})&if(x,y)\in [e,1]^2;\\
				&\textnormal{min}(x,y)&if(x,y)\in A(e).
			\end{aligned} \right.
		\end{align*}
		The set of uninorms as above will be denoted by $\mathscr{U}_{min}$.
		\item
		If $U(0, 1)=1$, then $\mathscr{U}_{max}$ has the same structure, changing minimum by maximum in $A(e)$.
	\end{enumerate}
\end{theorem}

\begin{definition}\label{def4}\cite{tnorm}
	A nullnorm is a binary function $F : [0, 1]^2 \rightarrow [0, 1]$, for all $x, y, z \in [0, 1]$, which satisfies the following conditions:
	
	\begin{enumerate}
		\item [(F1)]
		Commutativity: $F(x, y) = F(y, x)$;
		\item [(F2)]
		Associativity: $F(F(x, y), z) = F(x,F(y, z))$;
		\item [(F3)]
		Monotonicity: $F$ is non-decreasing in each place;
		\item [(F4)]
		Absorbing element: there exists an absorbing element $k \in [0, 1], F(k, x) = k$ and the following statements hold,
		$$F(0, x) = x \text{ for all }  x \le k,$$
		$$F(1, x) = x \text{ for all }  x \ge k.$$
	\end{enumerate} 	
\end{definition}

In general, $k$ is always given by $F(0,1)$.

In addition, some concepts related to fuzzy sets are also necessary.
\begin{definition}\cite{tnorm}
	Let $\mu$ be a mapping from the set $X$ to $[0,1]$, if  $$\mu:X\rightarrow [0,1], x\mapsto \mu(x),$$
	we call $\mu$ the fuzzy subset on $X$  and call $\mu(x)$  the membership function of the fuzzy subset  $\mu$.
	
\end{definition}

The set composed of all fuzzy subsets on $X$ is denoted as $\mathscr{F}(X)$. The classical set and the characteristic function are special cases of the fuzzy subset and the membership function respectively.

\begin{definition}\cite{VagueGroups}
	Let $X, Y$ be two sets, then the fuzzy relation $f$ from $X$ to $Y$ is a direct product $X\times Y=\{(x,y)|x\in X,y \in Y\}$ that is $f\in \mathscr{F}(X\times Y)$
	$$ f:X\times Y\rightarrow [0,1],$$
	where $f(x, y)$ represents the degree of element $x$ has a $f$ relationship with element $y$. Especially, when $X=Y$, $f$ is called a fuzzy relationship on $X$.
\end{definition}

Later we will use some knowledge about lattice.
\begin{definition}\cite{boundedlattice}
	A lattice is a nonempty set $L$ equipped with a partial order $\leq$ such that each two elements $x, y \in L$ have a greatest lower bound, called meet or infimum, denoted by $x \wedge y$, as well as a smallest upper bound, called join or supremum, denoted by $x \vee y$.
\end{definition}

If any number of elements have a smallest upper bound and a greatest lower bound in a lattice $L$, we say that $L$ is complete.  For $x, y\in L$, the symbol $x < y$ means that $x \leq y$ and $x \neq y$. If $x \leq y$ or $y < x$, then we say that $x$ and $y$ are comparable. Otherwise, we say that $x$ and $y$ are incomparable, which denoted as $x \| y$.

The inclusion of two fuzzy sets $\mu,\nu$ on $S$ is defined as follows: $\mu\subseteq \nu$ if and only if for all $x\in S,\mu(x)\leq \nu(x)$. Clearly the set of all fuzzy sets on $S$ is a complete lattice $\mathscr{L}$ under this ordering. We shall denote the supremum and infimum in $\mathscr{L}$ by $\cup$ and $\cap$.

\section{Fuzzy $t$-norm with some properties}\label{SEC3}

\subsection{Fuzzy groupoids and fuzzy groups}
We already know that a $t$-norm is a monoid, so the essence of defining fuzzy $t$-norms is to define fuzzy monoids. In order to introduce the definition of fuzzy $t$-norms, first we need to give the definition of fuzzy  groupoids and groups, which  are proposed by Rosenfeld\cite{Rosenfeld1971Fuzzysubgroups}. At the same time, in order to simplify the statement of the propositional proofs of some fuzzy $t$-norm in the future, some knowledge to be used is given as propositions together with the definitions of fuzzy groupoids and groups.

\begin{definition}\cite{Rosenfeld1971Fuzzysubgroups}                                  
	Let $S$ be a groupoid with binary operation $\circ$, $\mu$ is fuzzy subset of $S$. $\mu$ will be called a fuzzy subgroupoid of $S$ if, for all $x,y \in S$,
	$$\textnormal{min}(\mu(x),\mu(y))\leq \mu(x\circ y). $$
\end{definition}

The classic subgroupiod is actually a special case when fuzzy subset $\mu$ is the characteristic function.
\begin{proposition}
	Let $S$ be a groupoid with binary operation $\circ$, $\mu$ is fuzzy subset of $S$. Let $\mu$ be into \{0, l\}, so that $\mu$ is the characteristic function of a subset  $T \subset S$ if and only if $T$ is a subgroupoid.
\end{proposition}
\begin{proof}
	If $\mu$ is into \{0,1\}, then $\forall x,y\in S$, $\mu(x\circ y)\geq \textnormal{min}(\mu(x),\mu(y))$ is equivalent to
	$\mu(x)=\mu(y)=1$ implies $\mu(x\circ y)=1$, i.e., to $x, y \in T $ implies $x\circ y \in T$.
	
\end{proof}

\begin{example}
	
	Characteristic  functions of the empty set $\emptyset$ and the complete set $S$, $\mu_{id}(x)=0$, and  $\mu_{id}(x)= 1,x\in S$, can be fuzzy subgroupoids of any groupoid $S$, which is easy to verify by definition. We call them trivial fuzzy subgroupoids.
	
\end{example}

It can be found that the intersection of any number of fuzzy subgroupoids is still a fuzzy subgroupoid. The infimum of a set of fuzzy subsets $\{\mu_i\}$ is denoted as $\cap \mu_i$.
\begin{proposition}
	Let $S$ be a groupoid. The intersection of any number of fuzzy subgroupoids of $S$ is still a fuzzy subgroupoid of $S$.
\end{proposition}
\begin{proof}
	For a set of fuzzy subgroupoids $\{\mu_i \}$, $\forall x,y \in S$,
	\begin{equation*}
		\begin{aligned}
			(\cap \mu_i)(x\circ y)&=\textnormal{inf}[\mu_i(x \circ y)]\\
			&\geq \textnormal{inf}[\textnormal{min}( \mu_i(x),\mu_i(y) )]\\
			&=\textnormal{min}(\textnormal{inf} \mu_i(x),\textnormal{inf} \mu_i(y))\\
			&=\textnormal{min}((\cap \mu_i)(x),(\cap \mu_i)(y)).
		\end{aligned}
	\end{equation*}
\end{proof}

\begin{definition}\cite{Rosenfeld1971Fuzzysubgroups}                                      
	Let $S$ be a group with binary operation $\circ$ and identity element $e$. $\mu$ is a fuzzy subset on $S$. $\mu$ will be called a fuzzy subgroup of the group $S$ if $\mu$ satisfies:
	\begin{enumerate}[(1)]
		\item 	 $\textnormal{min}(\mu(x),\mu(y))\leq \mu(x \circ y), \forall x, y \in S. $
		\item	$\mu(x^{-1})\geq \mu(x),\forall x \in S.$
	\end{enumerate}
\end{definition}

\begin{proposition}
	$S$ is a group and $\mu$ is a fuzzy subset on $S$. If the range of $\mu$ is $\{0,1\}$, that is, $\mu$ is a characteristic function of some subset $T$ of $S$, then $\mu$ is the fuzzy subgroup of $S$ if and only if $T$ is a subgroup of $S$.
\end{proposition}
\begin{proof}
	First look at necessity. Since the previous proposition has already proved the closed property, here we only need to look at the associativity, the identity element and the inverse element.
	Associativity can be obtained by inheritance and closure of operation.
	Taking $x\in T$, so that $\mu(e)=\mu(x\circ x^{-1})\geq \mu(x)=1$, that is, the identity element $e\in T$. And for $\forall x\in T,\mu(x^{-1})\geq \mu(x)=1$, we have $\mu(x^{-1})=1$, that is, $x ^{-1}\in T$.
	
	Conversely, $T$ is a subgroup of $S$, then
	$$\forall x\in T,\mu(x^{-1})=\mu(x)=1,$$
	$$\forall x\notin T,\mu(x^{-1})=\mu(x)=0,$$
	
	so $\forall x\in S,\mu(x^{-1})\geq \mu(x)$.
	The rest is known from the previous propositions.
\end{proof}

It can be found that the intersection of any number fuzzy subgroups is still a fuzzy subgroup.
\begin{proposition}
	Let $S$ be a group, and the intersection of any number of fuzzy subgroups of $S$ is still a fuzzy subgroup of $S$.
\end{proposition}
\begin{proof}
	For a set of fuzzy subgroups $\{\mu_i \}$, $\forall x\in S$,
	\begin{equation*}
		\begin{aligned}
			(\cap \mu_i)(x^{-1})&=\textnormal{inf}(\mu_i(x^{-1}))\\
			&\geq \textnormal{\textnormal{inf}}(\mu_i(x))\\
			&=(\cap \mu_i)(x).
		\end{aligned}
	\end{equation*}
	The rest is known from the previous propositions.
\end{proof}

Now we present the definition of fuzzy monoids.
\begin{definition}\cite{Vagueandfuzzyt-normsandt-conorms}
	Let $G$ be a monoid with identity element $e$, $\mu$ a fuzzy subset of $G$. $\mu$ will be called  a fuzzy submonoid of $M$ if and only if
	\begin{enumerate}[(1)]
		\item 	 $\textnormal{min}(\mu(x),\mu(y))\leq \mu(x \circ y), \forall x, y \in G. $
		\item	$\mu(e) = 1$.
	\end{enumerate}	
\end{definition}

\subsection{Generalization of some properties of $t$-norm}
With the previous definition of fuzzy monoids, we can naturally define fuzzy $t$-norm, and then we generalize some properties originally belonging to $t$-norm.
\begin{definition}\label{def1}
	Let $G=([0,1], T)$ be a $t$-norm, $\mu$ a fuzzy subset of $G$. $\mu$ will be called  a fuzzy $t$-subnorm of $G$ if and only if
	\begin{enumerate}[(1)]
		\item 	 $\textnormal{min}(\mu(x),\mu(y))\leq \mu(T(x,y)), \forall x, y \in [0,1]. $
		\item	$\mu(1) = 1$.
	\end{enumerate}	
\end{definition}
\begin{proposition}
	Let $G=([0,1], T)$ be a $t$-norm, $\mu$  a fuzzy $t$-subnorm of $G$, then $\mu(T(x,y))=0 \Rightarrow \mu(x)=0$ or $\mu(y)=0$, $\forall x,y\in [0,1]$.
\end{proposition}
\begin{proof}
	By Definition \ref{def1}, it can be seen that  $\forall x,y\in [0,1]$, $$\textnormal{min}(\mu(x),\mu(y))\leq \mu(T(x,y))=0,$$
	so $\mu(x)=0$ or $\mu(y)=0$.
\end{proof}
Different $t$-norm may have different fuzzy $t$-subnorms.
\begin{example}
	The unique fuzzy $t$-subnorm of $T_M$ is $\mu_{id}(x)=x,x\in[0,1]$, since any other $t$-norm $T$ satisfy $T(x,y) < \textnormal{min}(x,y)$ for some $ (x,y)\in[0,1]^2$, so the fuzzy $t$-subnorm cannot be $\mu_{id}$, otherwise it will be contradiction with $\mu(T(x,y))\geq \textnormal{min}(\mu(x),\mu(y))=\textnormal{min}(x,y)$.
\end{example}

\begin{example}	
	$\mu(x)=1,x\in[0,1]$ can be a fuzzy $t$-subnorm of any $t$-norm, which is easy to verify by Definition \ref{def1}.
\end{example}

To better study fuzzy $t$-norms we try to extend the relevant properties of $t$-norms to it.
\begin{definition}
	Let $\mu$ be a fuzzy $t$-subnorm of a $t$-norm $T$, for any $ x,y,z\in [0,1]$, we consider the following properties:
	\begin{enumerate}[(1)]
		\item  $\mu$ is said to be fuzzy strictly monotone if
		$$\mu(T(x,y))>\mu(T(x,z)),~\text{$0<x<1,y<z$} .$$
		\item  $\mu$ satisfies the fuzzy cancellation law if
		$$\mu(T(x,y))=\mu(T(x,z)) \Rightarrow x=0 \text{ or } y=z.$$
		\item  $\mu$ satisfies the fuzzy conditional cancellation law if
		$$\mu(T(x,y))=\mu(T(x,z))>\mu(0) \Rightarrow \mu(y)=\mu(z).$$
		\item  $\mu$ is called Archimedean if
		$$\text{for each } x,y\in(0,1), \exists n\in \mathbb{N}, s.t.\mu(x_T^{(n)})<\mu(y).$$
		\item  $\mu$ has the fuzzy limit property if
		$$\text{for all}\ x\in(0,1), \lim\limits_{n \to \infty}\mu(x_T^{(n)})=\mu(0).$$
	\end{enumerate}
\end{definition}
\begin{remark}
	\quad
	\begin{enumerate}[(1)]
		\item The fuzzy strictly monotone direction of $t$-subnorm  $\mu$ here is opposite to the strictly monotone direction of the $t$-norm. Because $\textnormal{min}(\mu(x),\mu(y))\leq \mu(T(x,y))$ for any $ x,y\in [0,1]$, if following original definition, we have contraction: $\mu(T(x,y))<\mu(T(x,1)),\mu(T(x,y))<\mu(T(1,y))$.
		\item The case of $x=1$ in fuzzy strict monotonicity is also meaningless, since $z=1, \mu(T(x,z))=1$.
	\end{enumerate}
\end{remark}

In addition, we can get the fuzzy cancellation law from the fuzzy strict monotonicity, and the fuzzy conditional cancellation law can obtained from the fuzzy cancellation law, as in the case of the $t$-norm.
\begin{proposition}\label{pro3.6}
	If $\mu$ is fuzzy strictly monotone, then $\mu$ satisfies the fuzzy cancellation law.
\end{proposition}
\begin{proof}
	Assume $x,y,z\in [0,1]$. First, if $x>0,y<z$, then $\mu(T(x,y))>\mu(T(x,z))$.	
	Similarly, if $x>0,y>z$, then $\mu(T(x,y))<\mu(T(x,z))$.	
	So,  when $\mu(T(x,y))=\mu(T(x,z))$, we have $ x=0 $ or $ y=z$.
	
\end{proof}
\begin{proposition}\label{pro3.7}
	If $\mu$ satisfies the fuzzy cancellation law, then $\mu$ satisfies the fuzzy conditional cancellation law.
\end{proposition}
\begin{proof}
	Assume $x,y,z\in [0,1]$. If $\mu(T(x,y))=\mu(T(x,z))>\mu(0)$, we have $ x=0 \text{ or } y=z$ by the fuzzy cancellation law.
	If $x=0$, then $\mu(T(x,y))=\mu(T(x,z))=\mu(0)$ which is contradictory.
	
	In summary, $\mu(T(x,y))=\mu(T(x,z))>\mu(0) \Rightarrow y=z.$
\end{proof}

In addition, the strict monotonicity of $T$ has restrictions on the fuzzy $t$-subnorm $\mu$.
\begin{proposition}
	
	If the $t$-norm $T$ is strictly monotone and $\mu$ is a fuzzy $t$-subnorm of the $t$-norm $T$, then $\mu(x)$ cannot be strictly decreasing on $[0,1]$.
\end{proposition}
\begin{proof}
	The $t$-norm $T$ is strictly monotone, then for any $x,y\in (0,1)$,
	$$T(x,y)<T(x,1)=x,T(x,y)<y.$$
	If  $\mu(x)$ is strictly decreasing on $[0,1]$, then for all $x,y\in (0,1), \mu(T(x,y))>\text{max}(\mu(x),\mu(y))$, which contradicts the definition.
\end{proof}

However, ``$\mu(x)$ is not strictly decreasing on $[0,1]$" is not equal to ``the fuzzy $t$-subnorm $\mu$ is not fuzzy strictly  monotone". Furthermore, we can see that only when $t$-norm $T$ is strictly monotone, the fuzzy $t$-subnorm $\mu$ is likely to be fuzzy strictly  monotone.

\begin{proposition}
	If the $t$-norm $T$ is not strictly monotone, then the fuzzy $t$-subnorm is also not fuzzy strictly  monotone.
\end{proposition}
\begin{proof}
	If the $t$-norm $T$ is not strictly monotone, then there exists $ x>0$, $0\leq y<z\leq 1$, s.t. $T(x,y)=T(x,z)$. This means $\mu(T(x,y))=\mu(T(x,z))$.	Thus $\mu$ is not fuzzy strictly monotone.
\end{proof}

\section{Vague $t$-norm with some properties}\label{SEC4}

In this section, we first introduce some related concepts, and then generalize several properties of $t$-norm.
\subsection{Vague monoid}
The conjunctive operation $\wedge$ always stands for the minimum operation between two real numbers.
\begin{definition}\cite{VagueGroups}
	A mapping $E_X: X\times X \rightarrow [0,1]$ is called a fuzzy equality on $X$ if and only if it satisfy the following conditions:
	\begin{enumerate}[(1)]
		\item $E_X(x,y)=1 \Leftrightarrow x=y, \forall x,y\in X$.
		\item $E_X(x,y)=E_X(y,x), \forall x,y\in X$.
		\item $E_X(x,y)\wedge E_X(y,z)\leq E_X(x,z) , \forall x,y,z\in X$.
	\end{enumerate}
\end{definition}
For two non-mpty crisp sets $X$ and $Y$, let $E_X,E_Y$ be two fuzzy
equalities on $X$ and $Y$, respectively. Then a fuzzy relation $f$ on $X\times Y$ (a subset of $X\times Y$) is called a fuzzy function from $X$ to $Y$ w.r.t. $E_X$ and $E_Y$, denoted by the usual notation $f:X\rightarrow Y$, if and only if the characteristic function $\mu:X\times Y\rightarrow [0,1]$ of $f$ holds the following two conditions:
\begin{enumerate}[(1)]
	\item $\forall x\in X,\exists y\in Y,\mu(x,y)>0.$
	\item $\forall x,y \in X, \forall z,w\in Y,\mu(x,z)\wedge \mu(y,w)\wedge E_X(x,y)\leq E_Y(z,w).$
\end{enumerate}
A fuzzy function $f$ is called a strong fuzzy function if and only if it additionally
satisfies: $\forall x\in X,\exists y\in Y,\mu(x,y)=1$.

\begin{definition}\cite{Vagueandfuzzyt-normsandt-conorms}
	A strong fuzzy function $f:X\times X\rightarrow X$ w.r.t. a
	fuzzy equality $E_{X\times X}$ on $X\times X$ and a fuzzy equality $E_X$ on $X$ is said to be
	a vague binary operation on $X$ w.r.t. $E_{X\times X}$ and $E_X$.
	
\end{definition}
With vague binary operations, we can define vague semigroups and vague monoids.
\begin{definition}\cite{VagueGroups}
	$\circ$ is the vague binary operation on the fuzzy equivalence relation $E_{X\times X}$ and $E_X$ on the set $X$, $\mu:X\times X\times X\rightarrow [0,1]$ is its membership function. Then
	\begin{enumerate}[(1)]
		\item
		$X$ together with $\circ$, denoted by $(X,\circ)$, is called a vague semigroup if and only if the characteristic function $\mu$ fulfills the condition:
		\begin{align*}
			\mu(b,c,d)\wedge\mu(a,d,m)\wedge\mu(a,b,q)\wedge\mu(q,c,w)\leq E_X(m,w),
		\end{align*}
		where $\forall a,b,c,d,m,q,w\in X.$
		\item
		A vague semigroup $(X,\circ)$ is a vague monoid if and only if
		there exists an identity element $e\in X$ such that for every $a\in X$, $\mu(e,a,a)\wedge \mu(a,e,a)=1.$
		\item
		A vague monoid $(X,\circ)$ is a vague group if and only if for every $a\in X$, there exists an inverse  element $a^{-1}\in X$ such that$$\mu(a^{-1},a,e)\wedge \mu(a,a^{-1},e)=1.$$
		\item
		A vague semigroup $(X,\circ)$ is said to be abelian (commutative) if and only if $\circ$ satisfies the condition	$$\mu(a,b,m)\wedge \mu(b,a,w)\leq E_X(m,w),\forall a,b,m,w\in X.$$
	\end{enumerate}
\end{definition}
We know that in a group $G$, the left and right cancellation law holds, that is, $xy=xz\Rightarrow y=z,yx=zx\Rightarrow y=z,\forall x,y,z\in G$, this is because every element in the group has an inverse element, just multiply both sides of the equation by $x^{-1}$ to the left or right at the same time. We also have a generalized cancellation law in the vague group.

\begin{proposition}\cite{VagueGroups}
	Let $\circ$ be the vague binary operation on the fuzzy equality $E_{X\times X}$ and $E_X$ on the set $X$, $\mu:X\times X\times X\rightarrow [0,1]$ is its membership function. $(X,\circ)$ is vague group, then
	\begin{enumerate}
		\item[(1)]	$\mu(a,b,u)\wedge \mu(a,c,u)\leq E_X(b,c),\forall a,b,c,u\in X.$
		\item[(2)]  $\mu(b,a,u)\wedge \mu(c,a,u)\leq E_X(b,c),\forall a,b,c,u\in X.$
	\end{enumerate}
\end{proposition}
\begin{proof}
	\begin{enumerate}
		\item[(1)]	Since $\circ$ is a strong fuzzy function, $\forall a,b,c,u\in X$, exist $v\in X$ such that  $\mu(a^{-1},u,v)=1$.
		From  the definition of associativity in vague group, we have that
		
		$
		\mu(a,b,u)=\mu(a,b,u)\wedge\mu(a^{-1},u,v)\wedge \mu(a^{-1},a,e)\wedge \mu(e,b,b)
		\leq E_X(v,b) \text{ and }
		\mu(a,c,u)=\mu(a,c,u)\wedge\mu(a^{-1},u,v)\wedge \mu(a^{-1},a,e)\wedge \mu(e,c,c)
		\leq E_X(v,c).
		$
		
		Thus, $\mu(a,b,u)\wedge \mu(a,c,u)\leq E_X(b,v)\wedge E_X(v,c)\leq E_X(b,c)$.
		\item[(2)]  It can be proved in a similar way due to symmetry.
	\end{enumerate}
\end{proof}
\subsection{Vague $t$-norm with some extended properties}
If we slightly change the conditions in the definition of fuzzy equality, we get a fuzzy equality respect to $t$-norm.
\begin{definition}
	Let $T$ be a $t$-norm, $X$ is a set. A mapping $E_X:X\times X\rightarrow[0,1]$ will be called $T$-fuzzy equality if and only if the following conditions are satisfied:
	\begin{enumerate}[(1)]
		\item $E_X(x, x) = 1 ,\forall x\in X$
		\item $ E_X(x, y) = E_X(y, x),\forall x, y \in X$
		\item $ T(E_X(x, y),E_X(y, z)) \leq E_X(x, z),\forall x, y, z \in X.$
	\end{enumerate}
\end{definition}
If $\forall x,y\in X, E_X(x, y) = 1 \Rightarrow x = y$, then it is said that $T$-fuzzy equality $E_X$  separates points.

In order to obtain the vague $t$-norm, we give the definition of the $T$-vague binary operation similar to the vague binary operation.

\begin{definition}\cite{Vagueandfuzzyt-normsandt-conorms}
	Let $T$ be a $t$-norm, $E$ a $T$-fuzzy equality on $X$, then the a mapping $\tilde{\circ}:M\times M\times M\rightarrow [0,1]$ is a $T$-vague binary operation, if for all $ x,y,z,x',y',z'\in M$
	\begin{enumerate}[(1)]
		\item $T(\tilde{\circ}(x,y,z),E(x,x'),E(y,y'),E(z,z'))\leq \tilde{\circ}(x',y',z').$
		\item $T(\tilde{\circ}(x,y,z),\tilde{\circ}(x,y,z'))\leq E(z,z').$
		\item $\forall x,y\in M,\exists z\in M,\tilde{\circ}(x,y,z)=1.$
	\end{enumerate}
\end{definition}

\begin{definition}\cite{Vagueandfuzzyt-normsandt-conorms}
	Let $T$ be a $t$-norm, $E$ a $T$-fuzzy equality on a monoid $M$, $\tilde{\circ}$ is the $T$-vague binary operation on $M$ w.r.t. $E$, then $(M,\tilde{\circ})$ is a $T$-vague monoid if and only if it satisfies
	\begin{enumerate}[(1)]
		\item
		For any $ x,y,z,d,m,q,w\in M$,$$ T(\tilde{\circ}(y,z,d),\tilde{\circ}(x,d,m),\tilde{\circ}(x,y,q),\tilde{\circ}(q,z,w))\leq E(m,w).$$
		\item There exists an element $e\in M$ s.t. for each $a\in M$, $$T(\tilde{\circ}(e,a,a), \tilde{\circ}(a,e,a))=1.$$
	\end{enumerate}
\end{definition}
\begin{definition}\cite{Vagueandfuzzyt-normsandt-conorms}
	Let $T$ be a $t$-norm, $E$ a $T$-fuzzy equality on $[0,1]$, $\tilde{T}$ a $T$-vague binary operation on $[0,1]$  w.r.t. $E$. Then a $T$-vague monoid $([0,1],\tilde{T})$ is called a $T$-vague $t$-norm  if and only if
	$$E(T(x,y),z)=\tilde{T}(x,y,z).$$
\end{definition}
From the above definition, we show that the $T$-vague $t$-norm is commutative.
\begin{proposition}\cite{Vagueandfuzzyt-normsandt-conorms}
	Let $T$ be a $t$-norm, $E$ a $T$-fuzzy equality on $[0,1]$, $\tilde{T}$ a  $T$-vague binary operation on $[0,1]$  w.r.t. $E$. $([0,1],\tilde{T})$ is $T$-vague $t$-norm, then  it must be commutative, i.e., $T(\tilde{T}(a,b,m),\tilde{T}(b,a,w))\leq E(m,w). $
\end{proposition}
\begin{proof}
	From the definition of the $T$-vague equality and  the $T$-vague  $t$-norm, for all $a,b,m,w\in [0,1]$, we can conclude that
	\begin{equation*}
		\begin{aligned}
			T(\tilde{T}(a,b,m),\tilde{T}(b,a,w))&=T(E(T(a,b),m),E(T(b,a),w))\\
			&=T(E(T(a,b),m),E(T(a,b),w))\\
			&\leq E(m,w).
		\end{aligned}
	\end{equation*}
\end{proof}
Since the $T$-vague $t$-norm is commutative, we can think that the first two positions of the operation $\tilde{T}$ have the same status. Therefore, when we give the definitions of strict monotonicity and cancellation law, we only give the first one situation of the location.
\begin{definition}
	Let $T$ be a $t$-norm, $E$  a $T$-fuzzy equality on $[0,1]$,  $\tilde{T}$ a $T$-vague binary operation on $[0,1]$  w.r.t. $E$ and $([0,1],\tilde{T})$ a $T$-vague $t$-norm.
	\begin{enumerate}[(1)]
		\item $\tilde{T}$ is said to be vague strictly monotone if
		$$x<y,\tilde{T}(x,z,a)=\tilde{T}(y,z,b)\Rightarrow a<b,$$
		where $x, y, z, a, b \in [0,1]$.
		
		\item $\tilde{T}$ satisfies the vague cancellation law if
		$$\forall x,a,b,c \in[0,1],\tilde{T}(a,x,c)=\tilde{T}(b,x,c)\Rightarrow a=b.$$
	\end{enumerate}
\end{definition}

Similarly, we can get the vague  cancellation law from the vague strictly monotonicity as in the case of the $t$-norm.
\begin{proposition}\label{pro12}
	Let $T$ be a $t$-norm, $E$ a $T$-fuzzy equality on $[0,1]$,  $\tilde{T}$ a the $T$-vague binary operation on $[0,1]$  w.r.t. $E$ and $([0,1],\tilde{T})$ a $T$-vague $t$-norm. If $\tilde{T}$ is vague strictly monotone, then it must satisfy the vague cancellation law.
\end{proposition}
\begin{proof}
	If $\tilde{T}$ is vague strictly monotone,  then
	$$x<y\text{ and } \tilde{T}(x,z,a)=\tilde{T}(y,z,b)\Rightarrow a<b,$$
	where $x, y, z, a, b\in[0,1]$.
	Assume that $\tilde{T}(a,x,c)=\tilde{T}(b,x,c)\text{ and }a\neq b$,	if $a>b$, we have $c<c$, which conflicts with the definition of vague strictly monotone. And if $a<b$, we also have $c<c$. Therefore,
	$$\tilde{T}(a,x,c)=\tilde{T}(b,x,c)\Rightarrow a=b,\forall x,a,b,c \in [0,1],$$
\end{proof}

\section{T-norm on bounded lattice with some properties}\label{SEC5}
Finally, we generalize some conclusions obtained to a special partially ordered set, which called bounded lattice.

\subsection{Bounded lattice}
In this section, we recall some basic notions and results related to lattices and $t$-norms on a bounded lattice\cite{boundedlattice}.

A bounded lattice is a lattice $(L, \leq)$ which has the top element $1$ and the bottom element $0$, that is, two elements $1, 0\in L$ exist such that $0 \leq x \leq 1$ for all $x \in L$.

\begin{definition}
	Let $(L, \leq, 0, 1)$ be a bounded lattice and
	$a, b\in L$ with $a \leq b$. The subinterval $[a, b]$ is defined by
	$$ [a,b]=\{x\in L:a\leq x\leq b\}.$$
\end{definition}
Other subintervals such as $[a, b)$, $(a, b]$ and $(a, b)$ can be
defined similarly. Obviously, $([a, b], \leq)$ is a bounded lattice
with the top element $b$ and the bottom element $a$.
Let $(L, \leq, 0, 1)$ be a bounded lattice, $[a, b]$ be a subinterval
of $L$, $T_1$ and $T_2$ be two binary operations on $[a, b]^2$. If there
holds $T_1(x, y) \leq T_2(x, y)$ for all $(x, y) \in [a, b]^2$, then we say
that $T_1$ is less than or equal to $T_2$ or, equivalently, that $T_2$ is
greater than or equal to $T_1$, and written as $T_1 \leq T_2$.
\begin{definition}\cite{boundedlattice}
	Let $(L,\leq,0,1)$ be a bounded lattice and $[a,b]$ be a subinterval of $L$. A binary operation $T:[a,b]\times [a,b]\rightarrow [a,b]$ is said to be a $t$-norm on $[a,b]$ if for any $x,y,z\in [a,b]$, the following conditions are fulfilled:
	\begin{enumerate}[(1)]
		\item If $y\leq z$, then $T(x,y)\leq T(x,z)$,
		\item $T (x, T (y, z)) = T (T (x, y), z)$,
		\item $T(x,y)=T(y,x)$,
		\item $T(x,b)=x$.
	\end{enumerate}
\end{definition}
For the sake of brevity and without loss of generality, we set $a=0$ and $b=1$ in the above definition.
In order to be able to generalize the conclusions about fuzzy $t$-norm and vague $t$-norm to bounded lattices, we need to give the concept of fuzzy sets on lattices.
\begin{definition}\cite{tnorm}
	Let $L$ be a lattice and $X$ be a non-empty set. We call the mapping $A:X\rightarrow L$ a $L$-subset of $X$, and call $A(x)$($x\in X$) the degree of membership of $x$ to $A$, which is interpreted as the degree to which the object $x$ belongs to the $L$-subset $A$.
\end{definition}

\begin{definition}\cite{boundedlattice}
	Let $T$ be a $t$-norm on bounded lattice $G=(L,\leq,0,1)$, $\mu$ a $L$-subset of $[0,1]$. Then $\mu$ is the fuzzy $t$-subnorm of $G$ if and only if
	\begin{enumerate}[(1)]
		\item 	 $\mu(x)\wedge\mu(y)\leq \mu(T(x,y)), \forall x, y \in L. $
		\item	$\mu(1) = 1$.
	\end{enumerate}	
\end{definition}

\subsection{Generalization of the previous conclusions}

We have extended the strict monotonicity and cancellation laws of $t$-norm to the fuzzy and vague cases, and we can now do this work similarly for $t$-norm on bounded lattices.
\begin{definition}
	For a fuzzy $t$-subnorm $\mu$ of a bounded lattice $L$ with respect to a $t$-norm $T$, we consider the following properties:
	\begin{enumerate}[(1)]
		\item  $\mu$ is said to be fuzzy strictly monotone if
		$$ x\in L\backslash \{ 0,1 \}, y<z, \text{ then } \mu(T(x,y))>\mu(T(x,z)).$$
		\item  $\mu$ satisfies the fuzzy cancellation law if
		$$\mu(T(x,y))=\mu(T(x,z)) \Rightarrow x=0 \text{ or } y=z.$$
		\item  $\mu$ satisfies the fuzzy conditional cancellation law if		$$\mu(T(x,y))=\mu(T(x,z))>\mu(0) \Rightarrow \mu(y)=\mu(z).$$
		\item  $\mu$ is called fuzzy Archimedean if
		$$\forall x,y\in L\backslash \{ 0,1 \}, \text{ there exists }n\in \mathbb{N}, s.t.\ \mu(x_T^{(n)})<\mu(y).$$
		\item  $\mu$ has the fuzzy limit property if
		$$\forall x\in L\backslash \{ 0,1\}\text{, there exists} \lim\limits_{n \to \infty} \mu(x_T^{(n)})=\mu(0).$$
	\end{enumerate}
\end{definition}
We can also get the fuzzy cancellation law from the fuzzy strict monotonicity, and from the fuzzy cancellation law  deduce the  fuzzy conditional cancellation law, as in the case of the $t$-norms.
\begin{proposition}\label{13}
	If $\mu$ is fuzzy strictly monotone, then $\mu$ satisfies the fuzzy cancellation law.
\end{proposition}
\begin{proof}
	The proof is similar to Proposition \ref{pro3.6}.
\end{proof}
\begin{proposition}
	If $\mu$ satisfies the fuzzy cancellation law, then $\mu$ satisfies the fuzzy conditional cancellation law.
\end{proposition}
\begin{proof}
	The proof is similar to Proposition \ref{pro3.7}.
\end{proof}

\begin{definition}
	Let $T$ be a $t$-norm on a bounded lattice $L$, $X$  a set. The mapping $E_X:X\times X\rightarrow L$ is called $T$-fuzzy equality on $X$ with respect to $L$ if and only if the following conditions hold:
	\begin{enumerate}[(1)]
		\item $E_X(x, x) = 1 ,\forall x\in X.$
		\item $ E_X(x, y) = E_X(y, x),\forall x, y\in X.$
		\item $ T(E_X(x, y),E_X(y, z)) \leq E_X(x, z),\forall x, y, z \in X.$
	\end{enumerate}
\end{definition}
If  $E_X(x, y) = 1 \Rightarrow x = y$, then $T$-fuzzy equality $E_X$ seprates points.

\begin{definition}
	Let $T$ be a $t$-norm on a bounded lattice $L$, $E$ a $T$-fuzzy equality on $M$, then $\tilde{\circ}: M\times M\times M\rightarrow L$ is called a $T$-vague binary operation on $M$ if and only if  for all $x,y,z,x',y',z'\in M$
	\begin{enumerate}[(1)]
		\item $T(\tilde{\circ}(x,y,z),E(x,x'),E(y,y'),E(z,z'))\leq \tilde{\circ}(x',y',z').$
		\item $T(\tilde{\circ}(x,y,z),\tilde{\circ}(x,y,z'))\leq E(z,z').$
		\item $\forall x,y\in M,\exists z\in M,\tilde{\circ}(x,y,z)=1.$
	\end{enumerate}
\end{definition}

\begin{definition}
	Let $T$ be a $t$-norm on a bounded lattice $L$, $E$ a $T$-fuzzy equality on a monoid $M$, $\tilde{\circ}$  a  $T$-vague binary operation on $M$, then $(M,\tilde{\circ})$ is called $T$-vague monoid if and only if
	\begin{enumerate}[(1)]
		\item  For all $ x,y,z,d,m,q,w\in M$,$$ T(\tilde{\circ}(y,z,d),\tilde{\circ}(x,d,m),\tilde{\circ}(x,y,q),\tilde{\circ}(q,z,w))\leq E(m,w).$$
		\item  There exists an element $e\in M$ such that for each $a\in M$,
		$$T(\tilde{\circ}(e,a,a), \tilde{\circ}(a,e,a))=1.$$
	\end{enumerate}
\end{definition}
\begin{definition}
	Let $T$ be a $t$-norm on a bounded lattice $L$, $E$ a  $T$-fuzzy equality on a monoid $L$, $\tilde{T}$ a $T$-vague binary operation on $L$. A $T$-vague monoid  $([0,1],\tilde{T})$ is said to be  $T$-vague $t$-norm if and only if
	$$E(T(x,y),z)=\tilde{T}(x,y,z).$$
\end{definition}

\begin{definition}
	If $T$ is a $t$-norm on a bounded lattice $L$, $E$ is $T$-fuzzy equality on a monoid $L$, $\tilde{T}$ is a $T$-vague binary operation on $L$ and  $(L,\leq,0,1,\tilde{T})$ is $T$-vague $t$-norm.
	\begin{enumerate}[(1)]
		\item $\tilde{T}$ is vague strictly monotone if for all $x, y, z, a, b \in M$,
		$$x<y, \tilde{T}(x,z,a)=\tilde{T}(y,z,b) \text{, then }  a<b.$$
		\item $\tilde{T}$ satisfies vague cancellation law if
		$$\forall x,a,b,c \in M,\tilde{T}(a,x,c)=\tilde{T}(b,x,c)\Rightarrow a=b.$$
	\end{enumerate}
\end{definition}
\begin{proposition}\label{15}
	$T$ is a $t$-norm on a bounded lattice $L$, $E$ is $T$-fuzzy equality on a monoid $L$, $\tilde{T}$ is a $T$-vague binary operation on $L$ and  $(L,\leq,0,1,\tilde{T})$ is $T$-vague $t$-norm. If $\tilde{T}$ is vague strictly monotone, then $\tilde{T}$ satisfies vague cancellation law.
\end{proposition}
\begin{proof}
	The proof is similar to Proposition \ref{pro12}.
\end{proof}

\section{Other fuzzy methods}\label{SEC6}
In the definition of fuzzy monoid, the conjunctive operation $\wedge$ can be replaced by other operators, so as to get more fuzzy methods.

\begin{definition}\cite{Vagueandfuzzyt-normsandt-conorms}
	Let $(H, \ast)$ be a monoid, $T$ a $t$-norm, $e$ its identity element and a fuzzy subset of $H$. $\sigma$ is a $T$-fuzzy
	submonoid of $H$ is equivalent to the following conditions.
	\begin{enumerate}[(1)]
		\item
		$T(\sigma(a),\sigma(b)) \le \sigma(a\ast b)\quad  \forall a, b \in H.$
		\item
		$\sigma (e) = 1.$
	\end{enumerate}
\end{definition}

\subsection{Fuzzy submonoid about aggregation function}
In addition to $t$-norm, we can also use aggregation functions, uninorms, nullnorms and other operators to generate $\wedge$.
\begin{definition}
	Let $A$ be an aggregation operator, ($M$, $\circ$) a monoid, $e$ its identity element and $\mu$ a fuzzy subset of $M$. $\mu$ is a $A$-fuzzy submonoid of $M$ if and only if
	\begin{enumerate}[(1)]
		\item 	 $A(\mu(x_1),\cdots,\mu(x_n))\leq \mu(x_1 \circ \cdots \circ x_n), \forall x_1,\cdots,x_n \in M. $
		\item	$\mu(e) = 1$.
	\end{enumerate}
\end{definition}

\begin{proposition}\label{16}
	Let $(M,\circ)$ be a monoid and $\mu$ a $A$-fuzzy submonoid of $M$. Then the core $H$ of $\mu$ (i.e. the set of elements $x$ of $M$ such that $\mu(x) = 1$) is a submonoid of $M$.
\end{proposition}

\begin{proof}
	The identity element of $H$ obviously existed and the associativity is inherited.
	Let $x,y \in H$, then
	$$ 1\!= \!A(\mu(x),\mu(y),\mu(e),\cdots,\mu(e)) \leq \mu(x \circ y \circ e \circ \cdots \circ e)= \mu(x\circ y),$$
	therefore, $x \circ y \in H$.
\end{proof}

\begin{example}
	Let $A_{\textnormal{min}}(x_1,\cdots,x_n)=\textnormal{min}(x_1,\cdots,x_n)$. We have a fuzzy subset $\mu_1$, $\mu_1(x)=x$ ,which is the $A_{\textnormal{min}}$-fuzzy submonoid  of $([0,1],T_M)$ . Similarly, We have a fuzzy subset $\mu_2$, $\mu_2(x)=1-x$ ,which is the $A_{\textnormal{min}}$-fuzzy submonoid  of $([0,1],S_M)$. We can see the facts from following proposition.
\end{example}

In fuzzy logic, the unit interval with a $t$-norm or a $t$-conorm is the most important monoid. Thus we will consider fuzzy submonoids of a given $t$-norm or $t$-conorm.

\begin{definition}
	Let $A$ and $T$ be an aggregation operator and a $t$-norm, respectively. An $A$-fuzzy submonoid of $([0,1],T)$ will be called an $A$-fuzzy $t$-subnorm of $T$.
\end{definition}
\begin{definition}
	Let $A$ and $S$ be an aggregation operator and a $t$-conorm, respectively. An $A$-fuzzy submonoid of $([0,1],S)$ will be called an $A$-fuzzy $t$-subconorm of $S$.
\end{definition}

\begin{proposition}\label{17}
	Let  $A(x_1,\cdots,x_n)=\textnormal{min}(x_1,\cdots,x_n)$. $\mu$ is an $A$-fuzzy $t$-subnorm of $T_M$ if and only if $\mu(1)=1$.
\end{proposition}
\begin{proof}
	Every fuzzy subset of [0,1] satisfies
	$$ \textnormal{min}(\mu(x_1),\cdots,\mu(x_n))\leq \mu(\textnormal{min}(x_1,\cdots,x_n)) $$
	for all $x_1,\cdots,x_n \in [0,1]$.
\end{proof}

\begin{proposition}\label{18}
	Let  $A(x_1,\cdots,x_n)=\textnormal{min}(x_1,\cdots,x_n)$. $\mu$ is an $A$-fuzzy $t$-subconorm of $S_M$  if and only if $\mu(0)=1$.
\end{proposition}
\begin{proof}
	Every fuzzy subset of [0,1] satisfies
	$$ \textnormal{min}(\mu(x_1),\cdots,\mu(x_n))\leq \mu(\text{max}(x_1,\cdots,x_n)) $$
	for all $x_1,\cdots,x_n \in [0,1]$.
\end{proof}

\subsection{Fuzzy submonoid about uninorm}
\begin{definition}
	Let $U$ be a uninorm, ($M$, $\circ$) a monoid, $e$ its identity element and $\mu$ a fuzzy subset of $M$. $\mu$ is a $U$-fuzzy submonoid of $M$ if and only if
	\begin{enumerate}[(1)]
		\item 	 $U(\mu(x),\mu(y))\leq \mu(x \circ y), \forall x, y \in M. $
		\item	$\mu(e) = 1$.
	\end{enumerate}
\end{definition}
\begin{proposition}\label{19}
	Let $U$ be a uninorm, $(M,\circ)$ be a monoid and $\mu$ a $U$-fuzzy submonoid of $M$. Then the core $H$ of $\mu$ (i.e. the set of elements $x$ of $M$ such that $\mu(x) = 1$) is a submonoid of $M$.
\end{proposition}
\begin{proof}
	The identity element of $H$ obviously already existed and the associativity is inherited.
	Let $x,y \in H$, then
	$$ 1=U(\mu(x),\mu(y))\leq\mu(x\circ y)$$
	therefore, $x \circ y \in H$.
\end{proof}

\begin{definition}
	A discrete uninorm is a submonoid of a uninorm contanining 0 and 1.
\end{definition}
\begin{example}
	Let $U_L=<T_L,e,S_L>_{\textnormal{min}}, L_{n,m}=\{0,\frac{e}{n},\cdots,e,e+\frac{1-e}{m},\cdots,1 \} $, then $L_{n,m}$ is a discrete uninorm of $U_L$.
	
\end{example}
\begin{definition}\label{d:U-fuzzy_t-subnorm_T}
	Let $U$ and $T$ be a uninorm and a $t$-norm, respectively. A $U$-fuzzy submonoid of $([0,1],T)$ will be called a $U$-fuzzy $t$-subnorm of $T$.
\end{definition}
\begin{definition}\label{d:U-fuzzy_t-subconorm_S}
	Let $U$ and $S$ be a uninorm and a $t$-conorm, respectively. A $U$-fuzzy submonoid of $([0,1],S)$ will be called a $U$-fuzzy $t$-subconorm of $S$.
\end{definition}

\begin{theorem}\label{43}
	For a monoid $M$ with identity element $e$, if a uninorm $U$ is disjunctive, then $\mu$ is $U$-fuzzy submonoid of $(M,\circ)$ if and only if $\mu \equiv 1$.
\end{theorem}
\begin{proof}
	$U$ is disjunctive, so $U(0,1)=1$. Thus for every $x\in [0,1]$, we have $U(x,1)=1$.
	
	If $\mu$ is $U$-fuzzy submonoid of $([0,1],M)$, then $\mu(e)=1$ and
	$$ U(\mu(x),\mu(y))\leq \mu(x\circ y).$$
	For all $x \in [0,1]$, let $y=e$, we can see that $1=U(\mu(x),1)\leq \mu(x)$.
	In the other hand, if $\mu \equiv 1$, then  $\mu$ obviously is $U$-fuzzy submonoid of $([0,1],M)$.
\end{proof}

\begin{corollary}
	For a $t$-norm $T$, if a uninorm $U$ is disjunctive, then $\mu$ is $U$-fuzzy $t$-subnorm of ([0,1],\,$T$) if and only if $\mu \equiv 1$.
\end{corollary}

\begin{corollary}
	For a $t$-conorm $S$, if a uninorm $U$ is disjunctive, then $\mu$ is $U$-fuzzy  $t$-subconorm of $([0,1],S)$ if and only if $\mu\equiv 1$.
\end{corollary}

\begin{corollary}
	For arbitary uninorm $U$ in $\mathscr{U}_{max}$, if a fuzzy subset $\mu$ is $U$-fuzzy submonoid of $M$, then $\mu \equiv 1$.
\end{corollary}

\begin{proposition}\label{20}
	Let $U$ be a uninorm and $B$ be a set where $B=\{x\in[0,1]\,|\, \mu(x)\in [e,1]\}$.
	\begin{equation*}
		U(x,y)=\left\{
		\begin{array}{lcl}
			eT(\frac{x}{e},\frac{y}{e}),& &(x,y)\in[0,e]^2,\\
			e+(1-e)S_M(\frac{x-e}{1-e},\frac{y-e}{1-e}),	& &(x,y)\in (e,1]^2,\\
			\textnormal{min}(x,y),	& &otherwise.
		\end{array}
		\right.
	\end{equation*}
	where $\mu$ is a fuzzy subset. Furthermore, $\mu$ is $U$-fuzzy $t$-subnorm of $([0,1],T_M)$ if and only if $\mu$ is decreasing on $B$ and $\mu(1)=1$.
\end{proposition}
\begin{proof}
	We need to distinguish some cases:
	\begin{enumerate}[(1)]
		\item
		If $(\mu(x),\mu(y))\in [0,e]^2$, then one has that
		\begin{align*}
			U((\mu(x),\mu(y))) &=eT(\frac{\mu(x)}{e},\frac{\mu(y)}{e})\\
			&\leq \textnormal{min}(\mu(x),\mu(y))\\
			& \leq \mu(\textnormal{min}(x,y)).
		\end{align*}
		\item
		If $(\mu(x),\mu(y))\in [e,1]^2$, it follows that		
		$U((\mu(x), \mu(y)))=\text{max}(\mu(x),\mu(y)) \leq \mu(\textnormal{min}(x,y))$
		if and only if
		$x\leq y$, then $\mu(y)\leq \mu(x)$.
		\item
		If $(\mu(x),\mu(y))\in [0,e]\times [e,1]$, one concludes that\\
		$$U(\mu(x),\mu(y))=\textnormal{min}(\mu(x),\mu(y)) \leq \mu(\textnormal{min}(x,y)).$$	
	\end{enumerate}
\end{proof}
\begin{example}
	\begin{equation*}
		\mu(x)=\left \{
		\begin{array}{lcl}
			x,& &0\leq x\leq e,\\
			1,& &e\leq x \leq 1.
		\end{array}
		\right.	
	\end{equation*}
	
\end{example}
Not all fuzzy subset $\mu$ of a monoid $M$ can find a corresponding a uninorm $U$ such that  $\mu$ is a $U$-fuzzy submonoid of $M$. Then the following two propositions hold.
\begin{proposition}\label{21}
	Let $\mu(x)=x$ be a fuzzy subset, $T$ be a $t$-norm. There is no uninorm $U$ that $\mu$ is a $U$-fuzzy $t$-subnorm of $([0,1],T)$.
\end{proposition}
\begin{proof}
	Assume that there exists uninorm $U$ with identity element $e$ that $\mu$ is a $U$-fuzzy $t$-subnorm of $([0,1],T)$, then
	$$U(\mu(x),\mu(y))\leq \mu(T(x,y)).$$
	Furthermore, it holds that
	$$U(x,y)\leq T(x,y)\leq \textnormal{min}(x,y),$$
	which contradicts with the case of $x=e, y>e$.
\end{proof}
\begin{proposition}\label{22}
	Let $\mu(x)=1-x$ be a fuzzy subset, $S$ be a $t$-conorm. There is no uninorm $U$ that $\mu$ is a $U$-fuzzy $t$-subconorm of $([0,1],S)$.
\end{proposition}
\begin{proof}
	Assume that there exists uninorm $U$ with identity element $e$ that $\mu$ is a $U$-fuzzy $t$-subconorm of $([0,1],S)$, then
	$$U(\mu(x),\mu(y))\leq \mu(S(x,y)),$$
	Furthermore, it holds that
	$$U(1-x,1-y)\leq 1-S(x,y)\leq 1-\text{max}(x,y),$$
	which contradicts with the case of $x=1-e, y< 1-e$.
\end{proof}

\subsection{Fuzzy submonoid about nullnorm}
\begin{definition}
	Let $F$ be a nullnorm, $(M,\circ)$ a monoid, $e$ its identity element and $\mu$ a fuzzy subset of $M$. $\mu$ is a $F$-fuzzy submonoid of $M$ if and only if
	\begin{enumerate}[(1)]
		\item $F(\mu(x), \mu(y))\leq \mu(x\circ y)$, for every $x,y\in[0,1]$
		\item $\mu(e)=1$.
	\end{enumerate}
\end{definition}

\begin{proposition}\label{23}
	Let $(M,\circ)$ be a monoid and $\mu$ be a $F$-fuzzy submonoid of $M$. Then the core $H$ of $\mu $(i.e. the set of elements $x$ of $M$, such that $\mu(x)=1$) is a submonoid of $M$.
\end{proposition}
\begin{proof}
	The identity element of $H$ obviously existed and the associativity is inherited.
	Let $x,y \in H$, then
	$$ 1=F(\mu(x),\mu(y))\leq\mu(x\circ y),$$
	therefore, $x \circ y \in H$.
\end{proof}
\begin{definition}
	A discrete nullnorm is a submonoid of a nullnorm contanining 0 and 1.
\end{definition}
Let $S$ and $T$ be a $t$-conorm and a $t$-norm, $F=<S,k,T>$ denoted the nullnorms $F$ with absorbing element $k$ as follows
$$
F(x,y)=\left\{
\begin{array}{lcl}
	kS(\frac{x}{k},\frac{y}{k}),& &(x,y)\in[0,k]^2,\\
	(1-k)T(\frac{x-k}{1-k},\frac{y-k}{1-k})+k,	& &(x,y)\in (k,1]^2,\\
	k,	& &otherwise.
\end{array}
\right.
$$
\begin{example}
	Let $F_L $ be a nullnorm where $F_L=<S_L,k,T_L>, L_{n,m}=\{0,\frac{k}{n},\cdots,k,k+\frac{1-k}{m},\cdots,1 \} $, then $L_{n,m}$ is a discrete uninorm of $F_L$.
	
\end{example}
\begin{definition}
	Let $F$ and $T$ be a nullnorm and a $t$-norm, respectively. An $F$-fuzzy submonoid of $([0,1],T)$ will be called an $F$-fuzzy $t$-subnorm of $T$.
\end{definition}
\begin{definition}
	Let $F$ and $S$ be a uninorm and a $t$-conorm, respectively. An $F$-fuzzy submonoid of $([0,1],S)$ will be called an $F$-fuzzy $t$-subconorm of $S$.
\end{definition}

\begin{proposition}\label{24}
	If $\mu$ is a $F$-fuzzy submonoid of $M$ where $F$ is a nullnorm with absorbing element $k$ and $M$ is a monoid with identity element $e$, then $\mu(x)\geq k$ for every $x\in[0,1]$.
\end{proposition}
\begin{proof}
	If $\mu$ is a $F$-fuzzy submonoid of $M$, then we have $$F(\mu(x),\mu(y))\leq \mu(x\circ y),$$  $$\mu(e)=1.$$
	Let $y=e$, we obtain that
	$$k=F(\mu(x),k)\leq F(\mu(x),\mu(e))\leq \mu(x).$$
\end{proof}

\begin{corollary}
	If $\mu$ is a $F$-fuzzy $t$-subnorm of $T$ where $F$ is a nullnorm with absorbing element $k$ and $T$ is a $t$-norm, then $\mu(x)\geq k$ for every $x\in[0,1]$.
\end{corollary}

\begin{corollary}
	If $\mu$ is a $F$-fuzzy $t$-subconorm of $S$ where $F$ is a nullnorm with absorbing element $k$ and $S$ is a $t$-conorm, then $\mu(x)\geq k$ for every $x\in[0,1]$.
\end{corollary}

\begin{proposition}\label{25}
	A fuzzy subset $\mu$ is a $F_M$-fuzzy $t$-subnorm of $([0,1],T_{\textnormal{min}})$, where $F_M$ is a nullnorm with absorbing element $k$ if and only if $\mu(1)=1$, $\mu(x)\geq k $ for every $x\in[0,1]$.
\end{proposition}
\begin{proof}
	If  $\mu$ is a $F_M$-fuzzy $t$-subnorm of $([0,1],T_{\textnormal{min}})$, then we obviously have  $\mu(1)=1$ and $\mu(x)\geq k$ for every $x\in[0,1]$.
	When $(x,y)\in [0,e]^2$, $F_M =eS(\frac{x}{e},\frac{y}{e})\leq e$.
\end{proof}

Let $S$ and $T$ be a $t$-conorm and a $t$-norm, $F_M$ denoted the nullnorms $F$ with absorbing element $k$ as follows
$$
F(x,y)=\left\{
\begin{array}{lcl}
	kS(\frac{x}{k},\frac{y}{k}),& &(x,y)\in[0,k]^2,\\
	(1-k)T_M(\frac{x-k}{1-k},\frac{y-k}{1-k})+k,	& &(x,y)\in (k,1]^2,\\
	k,	& &otherwise.
\end{array}
\right.
$$

\begin{proposition}\label{F}
	A fuzzy subset $\mu$ is a $F_M$-fuzzy $t$-subnorm of $([0,1],T_M)$ where $F_M$ is a nullnorm with absorbing element $k$ if and only if $\mu(1)=1$, $\mu(x)\geq k $ for every $x\in[0,1]$.
\end{proposition}
\begin{proof}
	Firstly, if  $\mu$ is a $F_M$-fuzzy $t$-subnorm of $([0,1],T_{\textnormal{min}})$, then we obviously have  $\mu(1)=1$ and $\mu(x)\geq k$ for every $x\in[0,1]$.
	
	Conversely, if $\mu(1)=1$, $\mu(x)\geq k $ for every $x\in[0,1]$, then we need to categorize the discussion.
	\begin{enumerate}[(1)]
		\item   $(x,y)\in [0,k]^2$,\\ $F_M(\mu(x),\mu(y)) =kS(\frac{\mu(x)}{k},\frac{\mu(y)}{k})\leq k \leq \mu(T_M(x,y))$.
		\item $(x,y)\in [k,1]^2$, \\$F_M(\mu(x),\mu(y)) =	(1-k)T_M(\frac{\mu(x)-k}{1-k},\frac{\mu(y)-k}{1-k})+k =\textnormal{min}(\mu(x),\mu(y)) \leq \mu(T_M(x,y))$.
		\item In the other cases, $ F_M(\mu(x),\mu(y))\leq \mu(\mu(x),\mu(y))$ clearly holds.
	\end{enumerate}
\end{proof}
\begin{proposition}
	A fuzzy subset $\mu$ is a $F_M$-fuzzy $t$-subnorm of $([0,1],S_M)$, where $F_M$ is a nullnorm with absorbing element $k$ if and only if $\mu(0)=1$, $\mu(x)\geq k $ for every $x\in[0,1]$.
\end{proposition}
\begin{proof}
	The proof is similar to Proposition~\ref{F}.
\end{proof}

\section{Conclusion}\label{SEC7}
In this paper, we fuzzified the properties of $t$-norms, such as strict monotonicity, cancellation law, conditional law, Archimedean and limit property to fuzzy $t$-norms so that  we can analyze and classify the fuzzy $t$-norms.  Just as strict monotonicity leads to the cancellation law and the fuzzy cancellation law  deduce the  fuzzy conditional cancellation law  in the $t$-norms, we get the same results after extending the properties.  For the same purpose, vague properties of  strict monotonicity, cancellation law, conditional cancellation law are proposed for vague $t$-norms. And then, we generalize the related properties to bounded lattice.
In addition, we defined the concepts of fuzzy monoids by aggregate functions, uninorms and nullnorms. Then we analyze some of their features such as the structure of core and the constraints of fuzzy monoids.  At the same time, some important examples are analyzed to facilitate a more intuitive understanding.

The discussion in this article can be directly generalized for $t$-conorm. In the future, the fuzzified properties of $t$-norms and $t$-conorms will be further explored.

\section*{Acknowledgement}

This research was supported by the National Natural Science Foundation of China (Grant no. 12101500), the Chinese Universities Scientific Fund (Grant no. 2452018054) and the College Students' Innovation and Entrepreneurship Training Program (Grant no. S202010712009).


\end{document}